\documentclass[11pt,a4paper]{amsart}
\usepackage[T1]{fontenc}
\usepackage[latin9]{inputenc}
\usepackage{amssymb}
\usepackage{color}
\usepackage{graphicx,color}
\usepackage{amsmath, amssymb, graphics}
\makeatletter
\newtheorem{theorem}{Theorem}[section]

\theoremstyle{definition}

\theoremstyle{remark}
\newtheorem{remark}[theorem]{Remark}

\numberwithin{equation}{section}
\def\R{\mathbb{R}}

\numberwithin{equation}{section} 
\numberwithin{figure}{section} 
  \@ifundefined{theoremstyle}{\usepackage{amsthm}}{}
  \theoremstyle{plain}
  \theoremstyle{remark}
  \theoremstyle{definition}
\providecommand{\keywords}[1]
{
\small	
\textbf{\textit{Keywords. }} #1
}

\begin{document}

\title[Classification of geodesics on a cone in space]{Classification of geodesics on a cone in space}
\author[H\'ector Efr\'en Guerrero Mora]
{H\'ector Efr\'en Guerrero Mora}
\date{\today}

\address{Departamento de Matem\'aticas.
	Universidad del Cauca \\ Facultad de Ciencias Naturales Exactas y de la Educaci\'on. Popay\'an, Colombia}
\email{heguerrero@unicauca.edu.co}
\thanks{The author was supported in part by Universidad del Cauca project ID 4558.}
%
%

\begin{abstract}

This article presents a new way to classify geodesics on a cone in the Euclidean 3-space. This proof is obtained considering our main result, which establishes the necessary and sufficient conditions that a curve in space must satisfy: to be rectifying curve or that its trace is contained in the sphere centered at the origin.
\end{abstract}
\dedicatory{To my family.}
\maketitle

\section{INTRODUCTION.}
The geodesics curves on cones have been topics of interest in classical differential geometry. Curves have recently been defined in space, which help to classify the geodesics on the cone, between these curves are the rectifying curves and the slant helix.
The notion of rectifying curve has been introduced by Chen and is defined as a unit speed curve such that position vector always lies in its rectifying plane \cite{ChenBY:03}.
Bang-Yen Chen showed that a curve on a cone, not necessarily a circular one, is a geodesic if and only if it is either a rectifying curve or an open portion of a ruling \cite{BANG-Y-C:07}. In this article, we will give a different demonstration of this important result, based on the necessary and sufficient conditions that a curve in space must satisfy: to be rectifying curve or that its trace is contained in the sphere centered at the origin.
The notion of slant helix was introduced by Izuyama and Takeuchi and is defined as a unit speed curve and non-zero curvature such that the principal normal lines of it make a constant angle with a fixed direction \cite{IzTa:04}.\\
We show that a curve with curvature greater than zero in a right circular cone is a geodesic if and only if it is a rectifying curve and a slant helix at the same time.
Other researchers have recently obtained a family of rectifying slant helices in Euclidean 3-space that live in cones \cite{BAFKALK:05}. In this article we find, using a different path as the other researchers have done, the parametrization by arc length of the curve that satisfies the property of being a rectifying curve and slant helix at the same time.
\section{Some condition on the rectifying curves.}
We will consider the general cones as presented in \cite{BANG-Y-C:07}.
 A general cone with vertex at the origin $0\in \R^3$ is given as follows:\\
For a given curve $\textbf{y}=\textbf{y}(t)$ defined on an open interval $I$ lying on the unit sphere $S_{0}^2(1)$ centered at $0\in \R^3$. Let $C_{\textbf{y}}$ denoted the cone with vertex at $0\in \R^3$ over the curve $\textbf{y}$. \\$C_{\textbf{y}}$   can be parametrized as
\begin{equation*}
C_{\textbf{y}}(t,u)=u\textbf{y}(t), \  u\in\R^+.
\end{equation*} For $t_0\in I$, the curve $\beta(u)=C_{\textbf{y}}(t_0,u)$, $u\in \R^+$, is called a ruling (or generating half line).
  The cone $C_{\textbf{y}}$ is generated by the rulings.\\Now, given a curve $\alpha:I\rightarrow \R^3$, with curvature greater than zero, we will show that\\
$\alpha$ is a geodesic curve, on an arbitrary cone with vertex at the origin if and only if $\alpha$ is a rectifying curve. To achieve this result, we consider the following theorem.
\begin{theorem}\label{teoremauno}
Let $\alpha:I\rightarrow \R^3$ be a curve parametrized by arc length, with curvature greater than zero. Then  $\alpha$ is a rectifying curve \'{o} $\alpha(I)$ is contained in a sphere centered at the origin if and only if  $\alpha\times\frac{d\alpha}{ds}$ has constant magnitude nonzero.
\end{theorem}
\begin{proof}
If $\alpha$ is a rectifying curve, we can write \begin{equation*}\alpha(s)=(s+\frac{b}{a})\textbf{t}(s)+\frac{1}{a}\textbf{b}(s),\end{equation*}where $a$, $b$ are constants, $a\neq 0,$\cite{ChenBY:03}
\\ then
\begin{equation}\label{curvarectificante}
(\alpha\times\frac{d\alpha}{ds})(s)=[(s+\frac{b}{a})\textbf{t}(s)+\frac{1}{a}\textbf{b}(s)]\times\textbf{t}(s)=\frac{1}{a}
\textbf{n}(s),
\end{equation}
from this it follows that $\alpha\times\frac{d\alpha}{ds}$   has constant magnitude nonzero.\\
On the other hand, if $\alpha(I)$ is contained in a sphere centered at the origin, then $\mid\mid\alpha(s)\mid\mid$ is a constant and $\alpha(s)\cdot\textbf{t}(s)=0$, this implies that $\alpha\times\frac{d\alpha}{ds}$ has constant magnitude nonzero.\\
 Conversely, suppose that $\gamma=\alpha\times\frac{d\alpha}{ds}$ has constant magnitude nonzero.
  By differentiating  $\delta^2=(\alpha\times\frac{d\alpha}{ds}).(\alpha\times\frac{d\alpha}{ds})$, we obtain
  \begin{eqnarray*}
  0&=&2(\alpha\times\frac{d\alpha}{ds})'.(\alpha\times\frac{d\alpha}{ds})=2(\frac{d\alpha}{ds}\times\frac{d\alpha}{ds}+\alpha\times\frac{d^2\alpha}{ds^2}).(\alpha\times\frac{d\alpha}{ds})\\
  &=&2(\alpha\times\frac{d^2\alpha}{ds^2}).(\alpha\times\frac{d\alpha}{ds})\\
  &=&2\kappa_{\alpha}(\alpha\times\textbf{n}).(\alpha\times\textbf{t})=2\kappa_{\alpha}[(\alpha.\alpha)(\textbf{n}.\textbf{t})-(\alpha.\textbf{t})(\textbf{n}.\alpha)]\\
  &=&-2\kappa_{\alpha}(\alpha.\textbf{t})(\textbf{n}.\alpha).
  \end{eqnarray*}Since $\kappa_{\alpha}$ is nonzero, we have to \begin{equation*}\textbf{n}.\alpha=0\ \ \text{\'{o}}\ \  \alpha.\textbf{t}=0.\end{equation*}
  From this it follows that $\alpha$ is a rectifying curve \'{o} $\alpha(I)$ is contained in a sphere centered at the origin.
\end{proof}
\section{Geodesic curves on a cone with vertex at the origin}
  \begin{remark}
  Let $C_{\textbf{y}}$ be an arbitrary cone with vertex at the origin parametrized by $X(t,u)=u\textbf{y}(t)$, where $u\in \R^+$ and $\textbf{y}=\textbf{y}(t)$ is a curve on the unit sphere centered at the origin, which without loss of generality we can assume that $\textbf{y}(t)$ is a unit speed curve. Let $\alpha=\alpha(s)$ be a curve parametrized by arc length contained in a coordinate neighborhood of $X$, this is $\alpha(s)=X(t(s),u(s))$, for some $t=t(s)$ and $u=u(s)$ differentiable functions.
Then \begin{equation*}
\alpha\times\frac{d\alpha}{ds}=u^2\frac{dt}{ds}(\textbf{y}\times \frac{d\textbf{y}}{dt})
 =-u^2\frac{dt}{ds}\mid\mid (\frac{d\textbf{y}}{dt}\times\textbf{ y}) \mid\mid \textbf{N}=-u^2\frac{dt}{ds}\textbf{N},\end{equation*} where
 $\textbf{N}=\textbf{N}(t,u)$ is the normal vector of the cone along the curve $\alpha$.
  \end{remark}
\begin{theorem}\label{cono}
Let $\alpha:I\rightarrow \R^3$ be a curve, with curvature greater than zero.\\
$\alpha$ is a geodesic curve, on a general cone with vertex at the origin if and only if $\alpha$ is a rectifying curve.
\end{theorem}
\begin{proof}
 Let $\alpha$ be a geodesic curve with nonzero curvature parametrized by arc length, on the cone with vertex at the origin.
Then the normal vector $\textbf{N}$ is parallel to the normal vector $\textbf{n} $ of the curve $\alpha$ at the point $\alpha(s)$, this is $\textbf{N}=\pm\textbf{n}$. Therefore, we can write $\alpha\times\frac{d\alpha}{ds}=- u^2\frac{dt}{ds}\textbf{N}=\delta \textbf{n}$, for some differentiable function nonzero  $\delta=\delta(s)$.\\
Deriving the previous expression, we have $
\kappa_{\alpha}\alpha\times\textbf{n}=\frac{d\delta}{ds}\textbf{n}-\delta\kappa_{\alpha}\textbf{t}+\delta\tau_{\alpha}\textbf{b}$,
from this it follows that
\begin{equation*}
0=\kappa_{\alpha}<\alpha\times\textbf{n},\textbf{n}>=\frac{d\delta}{ds}<\textbf{n},\textbf{n}>-\delta\kappa_{\alpha}<\textbf{t},\textbf{n}>+\delta\tau_{\alpha}<\textbf{b},\textbf{n}>=\frac{d\delta}{ds},
\end{equation*}
therefore $\alpha\times\frac{d\alpha}{ds}=\delta_0 \textbf{n}$, for some nonzero constant $\delta_0=\mp u^2\frac{dt}{ds}$. This implies that $\alpha\times\frac{d\alpha}{ds}$ has constant magnitude nonzero, and by theorem \ref{teoremauno}, we have $\alpha$ is a rectifying curve \'{o} $\alpha(I)$ is contained in a sphere centered at the origin.\\
 If $\alpha$ is contained in a sphere centered at the origin, then $\alpha(s)=u_0\textbf{y}(t(s))$, by some nonzero constant $u_0.$ And this implies $u_0=\mid\delta_0\mid$, since the curve $\alpha$ is parametrized by arc length.\\ Now let $\kappa_g$ be the geodesic curvature of $\alpha$. Then
 \begin{eqnarray*}
 \kappa_g&=&<\frac{d\alpha'}{ds},\textbf{N}\times\frac{d\alpha}{ds}>=<u_0\frac{d^2t}{ds^2}\frac{d\textbf{y}}{dt}+u_0(\frac{dt}{ds})^2\frac{d^2\textbf{y}}{dt^2},(\frac{d\textbf{y}}{dt}\times \textbf{y})\times(u_0\frac{d\textbf{y}}{dt}\frac{dt}{ds})>\\
 &=&<u_0\frac{d^2t}{ds^2}\frac{d\textbf{y}}{dt}+u_0(\frac{dt}{ds})^2\frac{d^2\textbf{y}}{dt^2},u_0\frac{dt}{ds}\textbf{y}>
 =u_0^2(\frac{dt}{ds})^3<\frac{d^2\textbf{y}}{dt^2},\textbf{y}>=\frac{\pm1}{u_0},
 \end{eqnarray*}
 which is not given, since $\alpha$ is a geodesic, so $\alpha$ must be a rectifying curve.\\
 Reciprocally, suppose that $\alpha$ is a rectifying curve, taking into account the proof of theorem 3 of \cite{ChenBY:03}, without loss of generality we can assume that the curve is with arc length parameter.
  Then $\alpha$ is a curve, with curvature nonzero, given by
 \begin{equation}\label{parametrizacionrectificante}
 \alpha(s)=\frac{1}{a}\sqrt{1+(as+b)^2}\textbf{y}(c+\arctan{(as+b)})=u(s)\textbf{y}(t(s)),
 \end{equation}where $a$ and $b$ are constant, with $a$ positive and $\textbf{y}=\textbf{y}(t)$ is a curve with unit speed on the unit sphere centered at the origin.\\ Clearly $\alpha$ is contained in a cone with vertex at the origin and parametrized by $X(t,u)=u\textbf{y}(t),$ where $u\in\R^+$, so
 \begin{equation*} \alpha\times\frac{d\alpha}{ds}=\frac{1}{a}(\textbf{y}\times \frac{d\textbf{y}}{dt})=-\frac{1}{a}
\textbf{N},\end{equation*} where
 $\textbf{N}=\textbf{N}(t,u)$ is the normal vector of the cone along the curve $\alpha$. And
 taking into account the equation \ref{curvarectificante} that appears in the proof of the theorem \ref{teoremauno} we conclude that the normal vector $\textbf{n}$ of the curve $\alpha$ is parallel to the normal vector $\textbf{N}$ of the surface, this is $\alpha$ is a geodesic curve.
\end{proof}
Next we show that the only planar geodesics in a cone in space are portions of rulings.
\remark It is well known that a curve in $\R^3$ is called planar if it lives in a plane in $\R^3$, that is, if its torsion is zero.
\begin{theorem}
Let $\alpha$ be a geodesic curve on a cone, with vertex at the origin.\\
If $\alpha$ is a planar geodesic curve, then its curvature is equal to zero.
\end{theorem}
\begin{proof}
We can write the previous statement like this: If the curvature of $\alpha$ is greater than zero, then $\alpha$ is not a planar geodesic curve.
Thus by theorem \ref{cono} we have that $\alpha$ is a rectifying curve and by theorem 2 of \cite{ChenBY:03}, we have that its torsion must be different from zero. Therefore, $\alpha$ is not a planar geodesic curve.
\end{proof}
The following theorem allows us to classify the geodesic curves in a circular cone with vertex at the origin.
\begin{theorem}
Let $\alpha:I\rightarrow \R^3$ be a curve, with curvature greater than zero.
$\alpha$ is a geodesic curve on a circular cone with vertex at the origin if and only if $\alpha$ is a rectifying curve and it is at the same time slant helix.
\end{theorem}
\begin{proof}
By the previous theorem \ref{cono}, the two situations imply that $\alpha$ is a rectifying curve and therefore it can be written as: $ \alpha(s)=\frac{1}{a}\sqrt{1+(as+b)^2}\textbf{y}(t(s)),$ where $a$ and $b$ are constant, with $a$ positive and $\textbf{y}=\textbf{y}(t)$ is a curve with unit speed on the unit sphere centered at the origin.\\ Let us show that $\textbf{y}=\textbf{y}(t)$ forms a constant angle with a unit vector, fixed $U$ if and only if the normal vector $\textbf{n}=\textbf{n}(s)$ of the curve $\alpha$  also forms a constant angle with that same unit vector, fixed $U$.\\
Of course,
\begin{eqnarray*}
\frac{d}{ds}<\textbf{y}(t(s)),U>&=&\frac{d}{ds}<\frac{a\alpha(s)}{\sqrt{1+(as+b)^2}},U>\\
&=&a<\frac{(1+(as+b)^2)\textbf{t}(s)-a(as+b)\alpha(s)}{(1+(as+b)^2)^{3/2}},U>\\
&=&a<\frac{(1+(as+b)^2)\textbf{t}(s)-a(as+b)[(s+\frac{b}{a})\textbf{t}(s)+\frac{1}{a}\textbf{b}(s)]}{(1+(as+b)^2)^{3/2}},U>\\
&=&\frac{a<\textbf{t}(s),U>-a(as+b)<\textbf{b}(s),U>}{(1+(as+b)^2)^{3/2}},
\end{eqnarray*} Analogously, taking into account the  Frenet's equations and  theorem 2 of \cite{ChenBY:03}, we have
\begin{eqnarray*}
\frac{d}{ds}<\textbf{n}(s),U>&=&-\kappa(s)<\textbf{t}(s),U>+\tau(s)<\textbf{b}(s),U>\\
&=&-\kappa(s)<\textbf{t}(s),U>+\kappa(s)(as+b)<\textbf{b}(s),U>.
\end{eqnarray*}
This is:
\begin{equation}\label{Clasificación}
\frac{(1+(as+b)^2)^{3/2}}{a}\frac{d}{ds}<\textbf{y}(t(s)),U>=-\frac{1}{\kappa(s)}\frac{d}{ds}<\textbf{n}(s),U>,
\end{equation}
therefore
$\frac{d}{ds}<\textbf{y}(t(s)),U>=0$ if and only if $\frac{d}{ds}<\textbf{n}(s),U>=0$. And this proves the theorem.
\end{proof}
The following theorem shows us the parametrization by arc length of the
curve that satisfies the property of being a rectifying curve and slant helix at the
same time.
\begin{theorem}\label{TeoremaParametrización}
Let $\alpha$ be a curve parametrized by arc length, with curvature greater than zero.
If $\alpha$ is a rectifying curve and it is at the same time slant helix, then it can be parametrized as
\begin{equation*}
\alpha(s)=\frac{1}{a}\sqrt{1+(as+b)^2}(\sin\psi_0\cos{[\frac{c+\arctan{(as+b)}}{\sin\psi_0}]},
\sin\psi_0\sin{[\frac{c+\arctan{(as+b)}}{\sin\psi_0}]}, \cos\psi_0),
\end{equation*}
 where $a,b,c,\psi_0$ are real constants, $0<a,\ 0<\psi_0<\frac{\pi}{2}$.
\end{theorem}
\begin{proof}
The result is derived from the previous theorem, and considering the expression of the rectifying curve given by \ref{parametrizacionrectificante} , where \textbf{y} is the circumference defined as \begin{equation*}\textbf{y}(t)=(\sin\psi_0\cos{[\frac{t}{\sin\psi_0}]},
\sin\psi_0\sin{[\frac{t}{\sin\psi_0}]}, \cos\psi_0).\end{equation*}
\end{proof}
\section{Conclusion.}
\begin{enumerate}
\item The theorem \ref{teoremauno} which gives the necessary and sufficient conditions that a curve in space
 must satisfy: to be rectifying curve or that its trace is contained in the sphere centered at the origin,
 allows us to give a different proof of theorem 1 that appears in \cite{BANG-Y-C:07}.
\item The equation \ref{Clasificación} was important to classify geodesic curves in a circular cone with vertex at the origin. This equation also allows to classify
geodesics in arbitrary cones.
\end{enumerate}

\bibliographystyle{amsplain}

\end{document}